\documentclass[11pt, reqno]{article}
\usepackage{amsmath,amssymb,amsfonts,amsthm}
\usepackage{color}

\usepackage[colorlinks=true,linkcolor=blue,citecolor=blue,urlcolor=blue,pdfborder={0 0 0}]{hyperref}
\usepackage{cleveref}

\usepackage{caption}
\usepackage{thmtools}

\usepackage{mathrsfs}

\crefname{theorem}{Theorem}{Theorems}
\crefname{thm}{Theorem}{Theorems}
\crefname{lemma}{Lemma}{Lemmas}
\crefname{lem}{Lemma}{Lemmas}
\crefname{remark}{Remark}{Remarks}
\crefname{prop}{Proposition}{Propositions}
\crefname{defn}{Definition}{Definitions}
\crefname{corollary}{Corollary}{Corollaries}
\crefname{conjecture}{Conjecture}{Conjectures}
\crefname{question}{Question}{Questions}
\crefname{chapter}{Chapter}{Chapters}
\crefname{section}{Section}{Sections}
\crefname{figure}{Figure}{Figures}

\theoremstyle{plain}
\newtheorem{thm}{Theorem}[section]

\newtheorem{theorem}[thm]{Theorem}
\newtheorem{lemma}[thm]{Lemma}

\newtheorem{prop}[thm]{Proposition}
\newtheorem{proposition}[thm]{Proposition}
\newtheorem{conjecture}[thm]{Conjecture}

\theoremstyle{definition}

\theoremstyle{remark}

\numberwithin{equation}{section}

\renewcommand{\P}{\mathbb P}

\newcommand{\R}{\mathbb R}
\newcommand{\Z}{\mathbb Z}
\newcommand{\N}{\mathbb N}

\newcommand{\cO}{\mathcal O}

\newcommand{\sA}{\mathscr A}

\newcommand{\bbD}{\mathbb D}

\newcommand{\bbG}{\mathbb G}
\newcommand{\bbH}{\mathbb H}

\newcommand{\bbT}{\mathbb T}

\newcommand{\bbV}{\mathbb V}

\usepackage[margin=1.05in]{geometry}
\usepackage{mathtools}
\newcommand{\eps}{\varepsilon}
\usepackage{bbm}
\usepackage{setspace}
\usepackage{enumitem}
\usepackage{tikz}

\title{\textsc{Counterexamples for percolation on unimodular random graphs}}
\author{Omer Angel \and Tom Hutchcroft}
\AtEndDocument{
  \bigskip
  \small
  \textsc{Omer Angel} \par
  \textsc{Department of Mathematics, University of British Columbia} \par
  \textit{Email:} \texttt{angel@math.ubc.ca} \par
    \textsc{Tom Hutchcroft} \par
  \textsc{Statslab, DPMMS, University of Cambridge, and Trinity College, Cambridge} \par
  \textit{Email:} \texttt{t.hutchcroft@maths.cam.ac.uk} \par
  }

\begin{document}
             \maketitle

\begin{abstract}
We construct an example of a bounded degree, nonamenable, unimodular random rooted graph with $p_c=p_u$ for Bernoulli bond percolation, as well as an example of a bounded degree, unimodular random rooted graph with $p_c<1$ but with an infinite cluster at criticality. These examples show that two well-known conjectures of Benjamini and Schramm are false when generalised from transitive graphs to unimodular random rooted graphs. 
\end{abstract}

\section{Introduction}

In \textbf{Bernoulli bond percolation}, each edge of a connected, locally finite graph $G$ is chosen to be deleted randomly with probability $1-p$, independently of all other edges, to obtain a random subgraph $G[p]$ of $G$. When $G$ is infinite, the \textbf{critical parameter} is defined to be
% When $G$ is infinite, we are often interested in the geometry of $G[p]$ at and near the \textbf{critical parameter}
\[
p_c(G) = \inf\{ p\in [0,1] : G[p] \text{ contains an infinite connected component almost surely}\}
\]
and the \textbf{uniqueness threshold} is defined to be
\[
p_u(G) = \inf\{p \in [0,1] : G[p] \text{ contains a unique infinite connected component almost surely}\}.
\]
Traditionally, percolation was studied primarily on the hypercubic lattice $\Z^d$ and other Euclidean lattices.
 % and it remains a major open problem to establish that critical percolation on $\Z^d$ does not have an infinite connected component almost surely for every $d\geq 2$. 
% 
In their seminal paper \cite{bperc96}, Benjamini and Schramm proposed a systematic study of percolation on more general graphs, and
 posed many questions. They were particularly interested in \textbf{quasi-transitive} graphs, that is, graphs whose automorphism groups have only finitely many orbits.
Two central questions concern the existence or non-existence of infinite clusters at $p_c$, and the equality or inequality of $p_c$ and $p_u$. 
 They made the following conjectures. Specific instances of these conjectures, such as those concerning $\Z^d$, are much older.
 % They made the following conjectures. 

% In particular, we have the following conjectures of Benjamini and Schramm \cite{bperc96}; specific, still open, instances of these conjectures are much older. 

\begin{conjecture}
\label{conj1}
Let $G$ be a quasi-transitive graph, and suppose that $p_c(G)<1$. Then $G[p_c]$ does not contain an infinite cluster almost surely.
\end{conjecture}

\begin{conjecture}
\label{conj2}
Let $G$ be a quasi-transitive graph. Then $p_c(G)<p_u(G)$ if and only if $G$ is nonamenable.
\end{conjecture}

Given a set $K$ of vertices in a graph $G$, we define $\partial_E K$ to be the set of edges of $G$ that have exactly one endpoint in $K$. 
A graph is said to be \textbf{nonamenable} if 
\[
\inf\left\{ \frac{|\partial_E K|}{\sum_{v\in K} \deg(v)} : K \subseteq V \text{ finite}\right\}>0,
\]
and \textbf{amenable} otherwise. It follows from the work of Burton and Keane \cite{burton1989density} and Gandolfi, Keane and Newman \cite{gandolfi1992uniqueness} that $p_c(G)=p_u(G)$ for every amenable quasi-transitive graph, so that only the `if' direction of Conjecture \ref{conj2} remains open. It was also proven by H\"aggstr\"om, Peres, and Schonmann \cite{MR1676835,MR1676831,haggstrom1999percolation} that there is a unique infinite cluster for every $p>p_u$ when $G$ is quasi-transitive.
We refer the reader to \cite{grimmett2010percolation} for an account of what is known in the Euclidean case $G=\Z^d$, and to \cite{LP:book} for percolation on more general graphs.

% Traditionally, percolation was studied primarily on the hypercubic lattice $\Z^d$ and other Euclidean lattices, and it remains a major open problem to establish that critical percolation on $\Z^d$ does not have an infinite connected component almost surely for every $d\geq 2$. 

% In their seminal paper \cite{bperc96}, Benjamini and Schramm proposed a systematic study of percolation on more general graphs, and posed many questions. They were particularly interested in \textbf{quasi-transitive} graphs, that is, graphs whose automorphism groups have only finitely many orbits. In particular, they made the following conjecture.

Substantial progress on Conjecture \ref{conj1} was made in 1999 by Benjamini, Lyons, Peres, and Schramm~\cite{BLPS99b}, who proved that the conjecture is true for any \emph{nonamenable, unimodular} quasi-transitive graph. Here, a graph is unimodular if it satisfies the \emph{mass-transport principle}, see \cite[Chapter 8]{LP:book}. (More recently, the conjecture has been verified for all quasi-transitive graphs of \emph{exponential growth} \cite{timar2006percolation,Hutchcroft2016944}, and in particular for all nonamenable quasi-transitive graphs, without the assumption of unimodularity.)
% 
 % Major progress was also made on several on related problems in the theory of percolation on nonamenable graphs by H\"aggstr\"om, Peres, and Schonmann \cite{HaggPe99}, and Lyons and Schramm \cite{LS99}. 
% 
% 
In the mid 2000's, Aldous and Lyons \cite{AL07} showed that this result, as well as several other important results such as those of \cite{newman1981infinite,burton1989density,gandolfi1992uniqueness,BLS99,MR1676835,MR1676831,haggstrom1999percolation,LS99} 
% much of the theory of percolation on unimodular transitive graphs 
can be generalized, with minimal changes to the proofs, 
 to \emph{unimodular random rooted graphs}.
These graphs appear naturally in many applications: For example, the connected component at the origin in percolation on a unimodular transitive graph is itself a unimodular random rooted graph. 
An important caveat is that when working with unimodular random rooted graphs one should consider a different, weaker notion of nonamenability than the classical one, which we call \emph{invariant nonamenability} \cite[\S 8]{AL07}.

 % generalised\cref
% Substantial progress on Conjecture \ref{conj1}, along with many other problems in \cite{bperc96}, was made 
% Several other questions from \cite{bperc96} were solved
 % during a flurry of activity in the late 1990's and early 2000's; highlights include \cite{BLPS99b,LS99}

In this note, we construct examples to show that, in contrast to the situation for the classical results mentioned in the previous paragraph, Conjectures \ref{conj1} and \ref{conj2} are in fact both \emph{false} when generalized to unimodular random rooted graphs, even with the assumption of bounded degrees.

\begin{theorem}
\label{thm:critical}
There exists a bounded degree unimodular random rooted graph $(G,\rho)$ such that $p_c(G)<1$ but there is an infinite cluster $G[p_c]$ almost surely.
\end{theorem}

\begin{theorem}
\label{thm:pcpu}
There exists a unimodular random rooted graph $(G,\rho)$ such that $G$ has bounded degrees, is nonamenable, and has $p_c(G)=p_u(G)$ for Bernoulli bond percolation almost surely.
\end{theorem}

% Note that the example in \cref{thm:pcpu} is 
We stress that the example in Theorem \ref{thm:pcpu} is nonamenable in the classical sense (which is a stronger property than being invariantly nonamenable).
Thus, any successful approach to Conjectures 1 and 2 cannot rely solely on mass-transport arguments.  See \cite{beringer2016percolation} for some further examples of unimodular random rooted graphs with unusual properties for percolation, and \cite{1105.2638} for another related example.

\section{Basic constructions}

\subsection{Unimodularity and normalizability of unrooted graphs}

% A random rooted graph $(G,\rho)$ is said to be \textbf{unimodular} if 
% \[
% \E\left[\sum_{v\in V} f(\rho,v,G) \right] = \E\left[ \sum_{u\in V} f(u,\rho,G)\right]
% \]
% for every mass-transport $f$. 

We assume that the reader is familiar with the basic notions of unimodular random rooted graphs, referring them to \cite{AL07} otherwise. Since it will be important to us and is perhaps less widely known, we quickly recall the theory of unimodular random rooted graphs with fixed underlying graph from \cite[Section 3]{AL07}.

Let $G$ be a graph, let $\Gamma \subseteq \operatorname{Aut}(G)$ be a group of automorphisms of $G$, and for each $v\in  V$ let $\operatorname{Stab}_v = \{\gamma \in \Gamma : \gamma v = v\}$ be the stabilizer of $v$ in $\Gamma$.  The group $\Gamma$ is said to be \textbf{unimodular} if 
\[
|\operatorname{Stab}_v \gamma v|=|\operatorname{Stab}_{\gamma v} v|
\]
for every $v \in V$ and $\gamma \in \Gamma$, where $\operatorname{Stab}_v u$ is the orbit of $u$ under  $\operatorname{Stab}_v$.  The graph $G$ is said to be unimodular if $\operatorname{Aut}(G)$ is unimodular. 
Let $G$ be a connected, locally finite, unimodular graph and let $\cO$ be a set of \textbf{orbit representatives} of $\Gamma$. That is, $\cO \subseteq V$ is such that for every vertex $v\in V$, there exists exactly one vertex $o \in \cO$ such that $\gamma v = o$ for some $\gamma \in \Gamma$. 
We say that $(G,\Gamma)$ is \textbf{normalizable} if there exists a measure $\mu_G$ on $\cO$ such that if $\rho$ is distributed according to $\mu_G$ then the random rooted graph $(G,\rho)$ is unimodular. It is easily seen that the measure $\mu_G$ is unique when it exists.
 % that is supported on graphs isomorphic to $G$. 

It is proven in \cite[Theorem 3.1]{AL07} that a connected, locally finite, unimodular graph $G$ is normalizable if and only if
\[
Z_v(G)=\sum_{ o \in \cO} |\operatorname{Stab}_o(v)|^{-1} < \infty
\]
for some (and hence every) vertex $v \in V$, and moreover the measure $\mu_G$ can be expressed as
\[
\mu_G(\{o\})= Z_v(G)^{-1}|\operatorname{Stab}_o(v)|^{-1} \qquad o \in \cO.
\]
% then the resulting random rooted graph $(G,\rho)$ is unimodular. Moreover, the distribution of $(G,\rho)$ does not depend on the choice of $\cO$ and $v$, and it is the only unimodular random rooted graph that is almost surely isomorphic to $G$. 

\subsection{Building new examples from old via replacement}

We will frequently make use of the following construction, which allows us to construct one normalizable unimodular graph from another. 
Constructions of this form are well-known, see \cite{khezeli2017shift} and \cite{beringer2016percolation} for further background.

 Let $G=(V,E)$ be a connected, locally finite graph, let $\Gamma \subseteq \operatorname{Aut}(G)$ be a unimodular subgroup of automorphisms, and let $G'=(V',E')$ be a connected, locally finite graph. Let $M_1(V')$ be the set of functions $m: V' \to [0,1]$ with $|m|:=\sum_{v\in V'}m(v)<\infty$, and suppose that there exists a function $m:V\to M(V')$, $m : v \mapsto m_v$ such that
\begin{enumerate}
\item The functions $\{m_v : v\in V\}$ are a partition of unity on $V'$ in the sense that $\sum_{v\in V} m_v(u) = 1$ for every $u \in V'$.
 % and $m(u) \cap m(v) = \emptyset$ if $u\neq v$.
\end{enumerate}
and
\begin{enumerate}
\item[2.] $m$ is automorphism-equivariant on $V^2$ in the following sense: If $u,v,w,x\in V$ are such that $(w,x)=(\gamma u, \gamma v)$ for some $\gamma \in \operatorname{Aut}(G)$, then there exists an automorphism $\gamma'$ of $G'$ such that $(m_w,m_x) = (\gamma m_u, \gamma m_v)$.
% the multisets of isomorphism classes of rooted graphs
% $\{(G',w) : w\in m(v)\}$ and $\{(G',w) : w\in m(u)\}$ coincide.
\end{enumerate}
Then $G'$ is also unimodular. If furthermore $G$ is normalizable and 
% and that
\begin{equation}
\label{eq:m}
\sum_{o\in \cO(G)} \mu_G(\{o\}) |m_o|<\infty,
\end{equation}
then $G'$ is normalizable with
\[
\mu_{G'}(\{o'\}) = \sum_{v\in V'} \mathbbm{1}\left[o' \in \operatorname{Aut}\left(G'\right) v\right] \sum_{o\in \cO(G)}  \frac{\mu_G(\{o\})m_o(v)}{\sum_{o\in \cO(G)} \mu_G(\{o\}) |m_o|} \qquad o' \in \cO(G').
\]
%  Let $\tilde \rho$ be a random element of $\cO(G)$ drawn from the distribution $\tilde \mu_G$ defined by
% \[
% \tilde \mu_G(\{o\}) =  \frac{\mu_G(\{o\})|m(o)|}{\sum_{o\in \cO(G)} \mu_G(\{o\}) |m(o)|}, \qquad o \in \cO(G)
% \]
% and, conditional on $\tilde \rho$, let $\rho'$ be chosen according to the conditional distribution 
% \[
% \P(\rho'=u \mid \tilde \rho) = \frac{m_\rho(u)}{|m_\rho|}. 
% \]
%  Then the random rooted graph $(G',\rho')$ is unimodular, and so in particular $G'$ is normalizable. 

Following \cite{beringer2016percolation}, we call this method of constructing new normalizable unimodular graphs from old ones \textbf{replacement}. 
To give a simple example of replacement, suppose that $G$ is a connected, locally finite, normalizable unimodular graph, and let $G'$ be the graph in which each edge of $G$ is replaced with a path of length two. Define $m: V\to M_1(v')$ by setting $m_v(u)$ to be $1$ if $u$ is equal to $v$, and to be $1/2$ if $u$ is the midpoint of a path of length $2$ emanating from $v$ in $G'$ that was formerly an edge of $G$. It is easily verified that $m$ satisfies conditions $1$ and $2$ above. If furthermore $G$ has finite expected degree in the sense that $\sum_{o\in \cO(G)}\mu_G(\{o\})\deg(o)<\infty$, then \eqref{eq:m} is satisfied and  $G'$ is normalizable.

One can also consider a variation of this procedure allowing for randomization: 
% Suppose that $V'$ a set, that and that $G'$ is a 
 Let $G=(V,E)$ be a connected, locally finite, unimodular graph, let $V'$ be a set, and let $G'=(V',E')$ be a random connected, locally finite graph with vertex set $V'$, which we consider to be a random element of $\{0,1\}^{V^2}$. 
% If $\gamma':V'\to V'$ is a bijection, we say that the law of $G'$ is invariant under $\gamma'$ if for 
 Suppose that there exists a function $m:V\to M(V')$, $m : v \mapsto m_v$ such that
\begin{enumerate}
\item The functions $\{m_v : v\in V\}$ are a partition of unity on $V'$ in the sense that $\sum_{v\in V} m_v(u) = 1$ for every $u \in V'$.
 % and $m(u) \cap m(v) = \emptyset$ if $u\neq v$.
\end{enumerate}
and
\begin{enumerate}
\item[2.] $m$ is automorphism-equivariant on $V^2$ in the following sense: If $u,v,w,x\in V$ are such that $(w,x)=(\gamma u, \gamma v)$ for some $\gamma \in \operatorname{Aut}(G)$, then there exists a
bijection 
 $\gamma' : V' \to V'$ such that $(m_w,m_x) = (\gamma m_u, \gamma m_v)$ and the \emph{law} of $G'$ is invariant under the action of $\gamma'$ on $V^2$. 
% the multisets of isomorphism classes of rooted graphs
% $\{(G',w) : w\in m(v)\}$ and $\{(G',w) : w\in m(u)\}$ coincide.
\end{enumerate}
% Then $G'$ is also unimodular.
 Let $\tilde \rho$ be a random element of $\cO(G)$ drawn from biased measure $\tilde \mu_G$ defined by
\[
\tilde \mu_G(\{o\}) =  \frac{\mu_G(\{o\})|m(o)|}{\sum_{o\in \cO(G)} \mu_G(\{o\}) |m(o)|}, \qquad o \in \cO(G)
\]
and, conditional on $\tilde \rho$, let $\rho'\in V'$ be chosen according to the conditional distribution 
\[
\P(\rho'=u \mid \tilde \rho) = \frac{m_\rho(u)}{|m_\rho|}. 
\]
 Then the random rooted graph $(G',\rho')$ is unimodular.
 % , and so in particular $G'$ is normalizable. 

Fixing the vertex set of $G'$ in advance is of course rather unnecessary and restrictive, but it is sufficient for the examples we consider here.

\section{A discontinuous phase transition}

\subsection{Trees of tori}

Let $d \geq 2$. 
The \textbf{$d$-ary canopy tree} $T_d$ is the tree with vertex set $\Z \times \N$ and edge set
\[\left\{ \left\{ (i,j) , (k,j-1) \right\} : j \geq 1,\, d i \leq k \leq d (i+1)-1 \right\}.\]
In other words, $T_d$ is the tree that has infinitely many leaves (that have no children), and such that every vertex that is not a leaf has exactly $d$ children, that is, neighbours that are closer to the leaves than it is. 
Note that the isomorphism class of $(T_d,v)$ depends only on the distance between $v$ and the leaves, called the \textbf{height} of $v$, and denoted $|v|$. We also say that vertices with height $k$ for $k\geq 0$ are in \textbf{level} $k$. It is well known and easily verified  that $T_d$ is unimodular and normalizable, with
$\mu_{T_d}(\{o\}) = d^{-|o|+1}/(d-1).$

Let $n \geq 1$, and let $d,r \geq 2$.
We define the \textbf{tree of tori} $\mathbb{T}^n(d,r)$ to be the connected, locally finite graph with vertex set
\[ V\left(\mathbb{T}^n(d,r)\right) = \left\{ (v,x) : v \in V(T_d),\, x \in \Z^n / r^{|v|} \Z^n \right\}, \]
and where we connect two vertices $(v,x)$ and $(u,y)$ of $\mathbb{T}^n(d,r)$ by an edge if and only if either
\begin{enumerate}
\item $v=u$ and $x$ and $y$ are adjacent in the torus, or else
\item $u$ is adjacent to $v$ in $T_d$, and either $|v| \geq |u|$ and 
$x$ is mapped to $y$ by the quotient map $\Z^n / r^{|v|} \Z^n \to \Z^n / r^{|u|} \Z^n$ or, symmetrically, $|u| \geq |v|$ and 
$y$ is mapped to $x$ by the quotient map $\Z^n / r^{|u|} \Z^n \to \Z^n / r^{|v|} \Z^n$.
\end{enumerate}
See Figure \ref{fig:treeoftori} for an illustration. (Note that removing the torus edges from this graph yields the horocyclic product of the $d$-ary canopy tree with the $r^n$-ary tree, which also arises as a half-space of the Diestel-Leader graph $DL(d,r^n)$ \cite{MR1856226}.)

\begin{figure}[t]
\centering
% \begin{subfigure}[b]{\textwidth}
\includegraphics[width=0.95\textwidth]{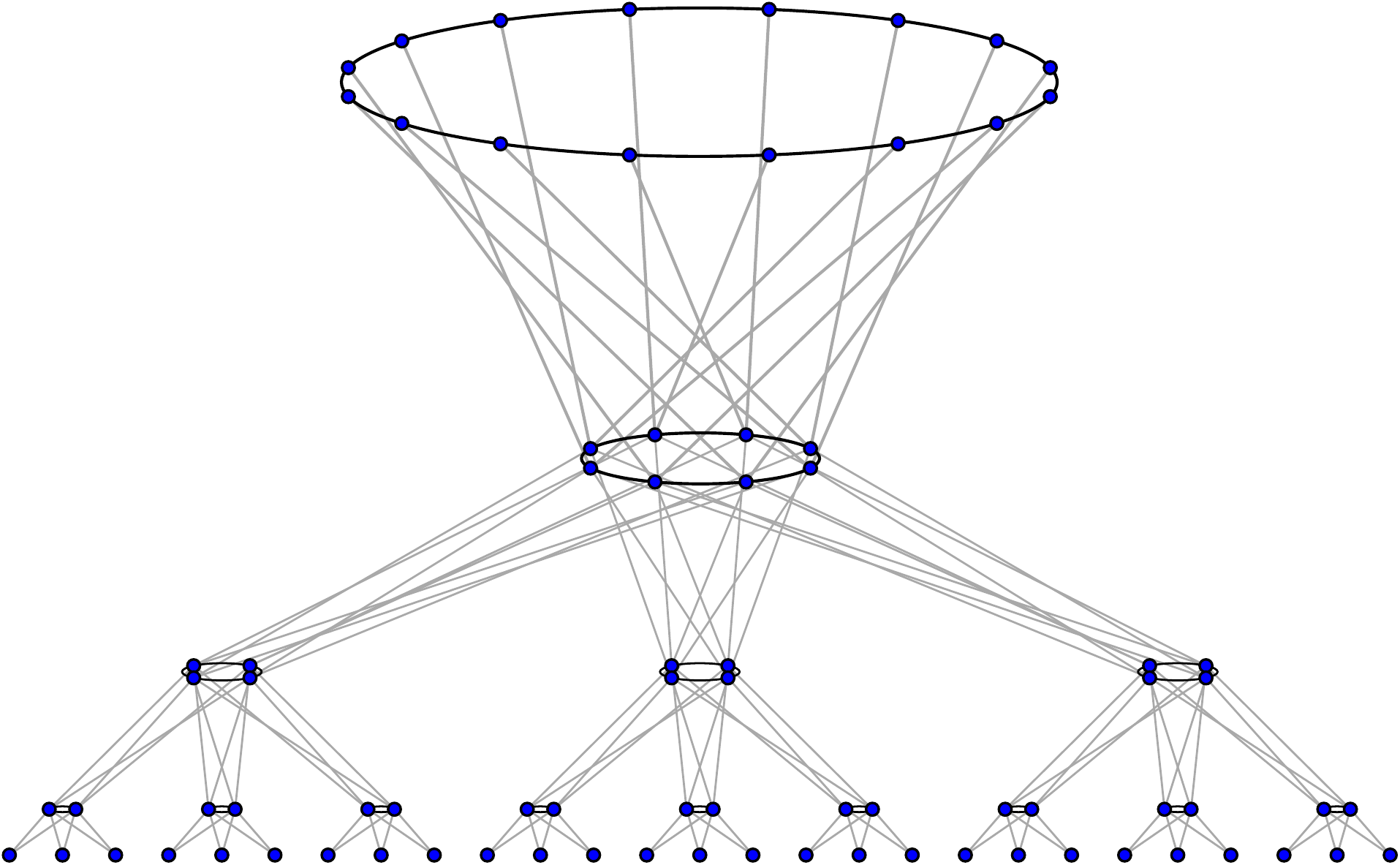}
\caption{The canopy tree of one-dimensional tori $\mathbb{T}^1(3,2)$. The grey edges give a $3$-to-$2$ correspondence between levels $l$ and $l+1$ for each $l\geq 0$. For general $\bbT^n(d,r)$ the correspondence is $d$-to-$r^n$. }
% \end{subfigure}
\label{fig:treeoftori}

\end{figure}

% Observe that, given $d$ and $r$, the isomorphism class of the rooted graph $\left(\mathbb{T}^n(d,r),(v,x)\right)$ depends only on $|v|$. 
The following is an easy consequence of replacement.
 % by replacement. 

\begin{proposition}
\label{prop:treesoftori}
% Suppose $d,n,$ and $r$ satisfy
% \[Z(d,n,r) = 1 + \sum_{\ell=1}^\infty r^{n \ell} \prod_{k=1}^\ell d(k)^{-1} < \infty.\]
If $r^n < d$ then $\bbT^n(d,r)$ is unimodular and normalizable, and
% $\rho=(\rho_1,\rho_2)$ be a random vertex of $\mathbb{T}^n(d,r)$ with 
\[\mu_{\bbT^n(d,r)}(\{o\}) = \frac{d^{-|o|+1}r^{n|o|}}{d-r^n}.\]
% \begin{cases}
% Z(d,n,r)^{-1}2^{n r(0)} & |o|=0\\
% Z(d,n,r)^{-1} 2^{n r(|o|)}\prod_{k=1}^{|o|} d(k)^{-1} & |o| \geq 1.
 % Then $(\mathbb{T}^n(d,r),\rho)$ is a unimodular random rooted graph.
\end{proposition}

It will also be useful to consider a more general version of this construction, in which we let the sizes of the tori grow as a specified function of the height. Let $d\geq 2$, let $n \geq 1$, and let $r:\N \to \N$ be an increasing function. We define the tree of tori $\tilde \bbT^n(d,r)$ similarly to above, with vertex set
\[ 
V\left(\tilde{\mathbb{T}}^n(d,r)\right) 
= \left\{ (v,x) : v \in V(T_d),\, x \in \Z^n / 
2^{ r(|v|)} \Z^n 
\right\}, \]
and where we connect two vertices $(v,x)$ and $(u,y)$ of $\mathbb{T}^n(d,r)$ by an edge if and only if either
\begin{enumerate}
\item $v=u$ and $x$ and $y$ are adjacent in the torus, or else
\item $u$ is adjacent to $v$ in $T_d$, and either $|v| \geq |u|$ and 
$x$ is mapped to $y$ by the quotient map $\Z^n / 2^{ r(|v|)} \Z^n \to \Z^n / 2^{r(|u|)} \Z^n$ or, symmetrically, $|u| \geq |v|$ and 
$y$ is mapped to $x$ by the quotient map $\Z^n / 2^{r(|u|)} \Z^n \to \Z^n / 2^{r(|v|)} \Z^n$.
\end{enumerate}
We now have, by replacement, that $\tilde\bbT^n(d,r)$ is unimodular and is normalizable if and only if
\[
\sum_{\ell \geq 0} d^{-\ell} 2^{n r(\ell)}<\infty, \quad \text{ in which case }\quad \mu_{\tilde{\bbT}^n(d,r)}(\{o\})= \frac{d^{-|o|+1}2^{nr(|o|)}}{\sum_{\ell \geq 0} d^{-\ell} 2^{n r(\ell)}}.
\]

\subsection{Proof of Theorem \ref{thm:critical}}

In this section we prove Theorem \ref{thm:critical}.
We begin with unbounded degree example, and then show how it can be modified to obtain a bounded degree example.

% Let $(T_2,\rho)$ be the unimodular random rooted binary canopy tree, in which every non-leaf vertex has degree $3$.
Let $d \geq 2$ and let $T_d$  be $d$-ary canopy tree. We write
\[
\log^+ x = \begin{cases}
1 & x \leq e\\
\log x &  x > e,
\end{cases}
\]
and write $\asymp$ for equalities that hold up to positive multiplicative constants. 
For each $\gamma \in \R$, let $G_{d,\gamma}$ be obtained from $T_d$ by replacing each edge connecting a vertex at height $n$ to a vertex at height $n+1$ with  
\[m_\gamma(n) := \left\lceil \frac{\log^+ n + \gamma \log^+ \log^+ n}{\log 2} \right\rceil\]  parallel edges, which is chosen so that
\[\left(\frac{1}{2}\right)^{m_\gamma(n)} \asymp \frac{1}{n \log^\gamma n}.\] 
% Here, $m_\gamma$ is chosen so that
It follows by replacement that $G_{d,\gamma}$ is unimodular. 

The basic idea behind this construction is that the coefficient of $\log^+ n$ above determines the value of $p_c$ (set here to be $1/2$), while the coefficient of $\log^+ \log^+ n$ determines the behaviour of percolation \emph{at} $p_c$.

\begin{proposition}
$p_c(G_{d,\gamma})=1/2$ for every $\gamma \in \R$. If $\gamma>1$, then  critical percolation on $G_{d,\gamma}$ contains an infinite cluster almost surely.
\end{proposition}

\begin{proof}
If $v$ is a vertex of $G_{d,\gamma}$, then the cluster of $v$ in $G_{d,\gamma}[p]$ is infinite if and only if every ancestor of $v$ has an open edge connecting it to its parent. This event occurs with probability
% Thus, the probability that a vertex $v$ of $G_m$ is in an infinite connected cluster of Bernoulli $p$ percolation is equal to
\[ \theta_\gamma(v,p) = \prod_{n \geq |v|} \left(1-(1-p)^{m_{\gamma}(n)}\right).\]
In particular, $\theta_\gamma(v,p)>0$ if and only if 
\[
  \sum_{n \geq 0} (1-p)^{m_\gamma(n)} < \infty.
\]
Since
\[ 
  n^\frac{\log p}{\log 2}  (\log n)^\frac{\gamma \log p}{\log 2}
  \leq (1-p)^{m_\gamma(n)}
  \leq (1-p) n^\frac{\log p}{\log 2}  (\log n)^\frac{\gamma \log p}{\log 2},
\]
it follows that $\theta_\gamma(v,p)>0$ if and only if
\[\sum_{n \geq 0} n^\frac{\log p}{\log 2}  (\log n)^\frac{\gamma \log p}{\log 2} <\infty.
\]
Recall that the series 
\[
  \sum_{n \geq 0} \frac{1}{n^\alpha (\log^+ n)^\beta} 
\]
converges if and only if either $\alpha>1$ or $\alpha=1$ and $\beta>1$.
Thus, $\theta_\gamma(v,p)>0$ for some (and hence every) vertex $v$ of $G_m$ if and only if either $p>1/2$, or $p=1/2$ and $\gamma>1$.
In particular, $p_c(G_{d,\gamma})=1/2$ for every $d$ and every value of $\gamma$, while if $\gamma>1$ then $\theta_{\gamma}(v,1/2)>0$ for every vertex $v$ of $G$ as desired.
\end{proof}

% We now construct a bounded degree variant of this example. Let $(T_d,\rho)$ be a unimodular canopy tree as before. This time, we 
% \medskip

% Let $d: \N \setminus \{0\} \to \N \setminus \{0\}$ 

\medskip

We now build a bounded degree variation on this example using trees of tori. Let $d,r \geq 2$, be such that $d > r^2$, and let 
 % : \N\setminus \{0\} \to \N\setminus \{0\}$, $r: \N \to \N$ and 
$m : \N \to \N \setminus\{0\}$ be a function. 
% , and that $r$ is increasing. 
 % \[1 + \sum_{\ell=1}^\infty m(\ell) 2^{n r(\ell)} \prod_{k=1}^\ell d(k)^{-1}\]
% be such that
% 
 % $\mathbb{T}^2(d,r)$ be the corresponding canopy tree of tori. Let $m : \N\setminus\{0\} \to \N \setminus\{0\}$ be such that 
% 
Let $\mathbb{G}(d,r,m)$ be the graph obtained by replacing each edge connecting two vertices of height $\ell$ and $\ell+1$ in $\mathbb{T}^2(d,r)$ with a path of length $m(\ell)$. It follows by replacement that $G$ is unimodular, and is normalizable if
\begin{equation}
\label{eq:pcnormalization}
\sum_{\ell \geq 0} d^{-\ell} r^{2 \ell} m(\ell) <\infty.
% \sum_{\ell=1}^\infty \left[2^{n r(\ell)}+m(\ell) 2^{n r(\ell+1)}\right] \prod_{k=1}^\ell d(k)^{-1} <\infty.
\end{equation}
Thus, Theorem \ref{thm:critical} follows immediately from the following proposition.
\begin{proposition}
% \cref{corollary:quantitativeslow}
\label{prop:criticalbdddegree}
Let $0<q<1$ be sufficiently large that $\theta_q(\Z^2)>3/4$, and 
% Let $d,r \geq 1$ be such that $d>r^2$ and let
% \[
% 1-(3/4)^{m(\ell)} =  1
% \]
let $m:\N \setminus\{0\}\to\N \setminus \{0\}$ be such that there exists a positive constant $c$ such that
\[
c\, 4^{-\ell}(\ell+1)^2 \leq  q^{m(\ell)} \leq 4^{-\ell}(\ell+1)^2
\]
for every $\ell \geq 1$. 
Then $\bbG=\bbG(5,2,m)$ is a normalizable, bounded degree, unimodular graph, $p_c(\bbG)=q$, and $\bbG[q]$ contains an infinite cluster almost surely.
\end{proposition}

For an example of a function $m$ of the form required by Proposition \ref{prop:criticalbdddegree}, we can take
\[
m (\ell) \equiv \left\lceil 
% \frac{\log \ell + \gamma \log \log \ell}{}
\frac{(\ell+2) \log 4 - 2 \log^+ \ell}{\log (1/q)}
\right\rceil.
\]

\begin{proof}
It is clear that $\bbG=\mathbb{G}(5,2,m)$ has bounded degrees, and we have already established that it is unimodular and normalizable.
% Let us first prove that
We now prove the statements concerning percolation on $\bbG$. Suppose that $v$ is a vertex of the canopy tree $T_d$ with $|v|=\ell$, and suppose that $u$ is the parent of $v$ in $T_d$. Thus, the torus $\{v\}\times \left(\Z^2/2^\ell \Z^2\right)$ is connected in $\bbG$ to the torus $\{u\}\times \left(\Z^2/2^{\ell+1} \Z^2\right)$ by 
$4^{\ell+1}$ paths of length $m(\ell)$. 
If $p<q$ then $p=q^{1+\delta}$ for some $\delta>0$, and so
the expected number of these paths that are open in $\bbG[p]$ is 
\[ 4^{\ell+1} p^{m(\ell)} 
= 4^{\ell+1}q^{(1+\delta)m(\ell)}
\asymp \ell^2 4^{-\delta \ell}.\]
 % \exp\left[ \left(\log 4 + \frac{\log 4 \log p}{\log(1/q)}\right) \ell - \frac{2 \log p}{\log(1/q)} \log \ell \right]\]
Since this expectation converges to zero, it follows that $\bbG[p]$ does not contain an infinite cluster almost surely, and we conclude that $p_c(\bbG) \geq q$. 

It remains to prove that $\bbG[q]$ contains an infinite cluster almost surely. Broadly speaking, the idea is that, since $\theta_q(\Z^2)>3/4$, each torus in $\bbG$ has a high probability to contain a giant open component which contains at least three quarters of its vertices, which is necessarily unique. The logarithmic correction in the definition of $m$ then ensures that the giant component in each torus is very likely to be connected by an open path to the giant component in its parent torus, which implies that an infinite open component exists as claimed. 

To make this argument rigorous, we will apply the following rather crude  estimate. 
% Presumably much better bounds can be found in the literature. 

\begin{proposition}
\label{lem:supercriticaltoruspercolation}
Consider Bernoulli bond percolation on the $n \times n$ torus, $\Z^2/n\Z^2$, for $p>p_c(\Z^2)$ supercritical. There exist positive constants $c_1$ and $c_2$ depending on $p$ such that for every $\eps>0$, the probability that $\Z^2/n\Z^2$ does \emph{not} contain an open cluster $C$ with $|C| \geq (\theta_{p}(\Z^2)-\eps)n^2$ is at most
 % and let $\sA$ be the event that $\Z^2/n\Z^2$ contains an open cluster $C$ such that $|C| \geq (\theta_{p}(\Z^2)-\eps)n^2$.  
 % Then there exists a constant $c=c_p$ such that
\[\frac{c_1 }{\eps^2 n^2} + n^2e^{-c_2 \eps n}.\]
\end{proposition}

\begin{proof}
It suffices to prove the analogous statement for the box $[1,n]^2$, which we consider as a subgraph of $\Z^2$.
It follows from \cite[Theorem 1.1]{AntalPisztora96} that if $p>p_c(\Z^2)$, $\delta>0$,  and $x,y \in [\delta n , (1-\delta)n]^2$, then there exists a positive constant $c_p$ such that
\[
\P\left( x \leftrightarrow \infty \text{ and } y \leftrightarrow \infty, \text{ but } x \nleftrightarrow y \text{ in } [0,n]^2 \cap \Z^2 \right)\leq e^{-c_p \delta n}.
\]
Thus, it follows by a union bound that the probability that the largest cluster in $[1,n]^2$ has size at most $(\theta_p(\Z^2)-\eps)n^2$ is at most
\[
\P_p\left(\sum_{x,y \in [\delta n , (1-\delta)n]^2 } \mathbbm{1}\left(x\leftrightarrow \infty, y \leftrightarrow \infty \right) \leq 
(\theta_p(\Z^2)-\eps)n^2 \right) + n^2 e^{-c_p \delta n}.
\]
%  We apply the following theorem of [ref]: If $p>p_c$ then there exists $\xi_p>0$ such that
% \[
% \P_p( v \leftrightarrow \partial B_m(v),\, v\nleftrightarrow \infty ) \leq e^{-\xi_p m}.
% \] 
% This yields the correlation estimate
% \begin{align*}
% \P_p(u \leftrightarrow \infty, v \leftrightarrow \infty)-\P_p(v \leftrightarrow \infty)^2 &\leq \P_p\left(u \leftrightarrow \partial B_{d(u,v)/4}(u), v \leftrightarrow \partial B_{d(u,v)/4}(v)\right) - \P_p(v \leftrightarrow \infty)^2\\
% &= \P_p\left(u \leftrightarrow \partial B_{d(u,v)/4}(u)\right)^2 - \P_p(v \leftrightarrow \infty)^2\\
% &\leq \left(\P_p(v \leftrightarrow \infty) + e^{-\xi_p d(u,v)/4} \right)^2 - \P_p(v \leftrightarrow \infty)^2 \leq 3e^{ -\xi_p d(u,v)/4 }.
% \end{align*}
% Summing over $[1,n]^2$ we obtain that
On the other hand, we have that \cite[Section 11.6]{grimmett2010percolation}
\[
\operatorname{Var}\left[ \sum_{x\in [\delta n, (1-\delta) n]^2} \mathbbm{1}\left( x \leftrightarrow \infty \right)\right] \leq Cn^2
\]
for some constant $C=C_p$, and it follows by Chebyshev's inequality that
\[
\P\left(\sum_{x\in [\eps n/2 , (1-\eps/2)n]^2} \mathbbm{1}\left( x \leftrightarrow \infty \right) \leq (\theta_p(\Z^2)-\eps)n^2\right) \leq C\left[(1-\delta)^2 \theta_p(\Z^2)-\theta_p(\Z^2) +\eps\right]^{-2} n^{-2}
\]
when the right hand side is positive.
We conclude by taking $\delta>0$ so that $(1-\delta)^2 \theta_p(\Z^2)-\theta_p(\Z^2) +\eps = \eps/2$.
\end{proof}

We now apply Lemma \ref{lem:supercriticaltoruspercolation} to complete the proof of Proposition \ref{prop:criticalbdddegree}. 
Let $v_0$ be a leaf of $T_5$, let $v_1,v_2,\ldots$ be its sequence of ancestors, and let $\Lambda_i$ be the torus $\{v_i\} \times \left(\Z^2/ 2^i \Z^2\right)$ in $\bbG$. It follows from Lemma \ref{lem:supercriticaltoruspercolation} and the Borel-Cantelli Lemma that $\Lambda_i[q]$ contains a (necessarily unique) giant open cluster of size at least $(3/4)4^i$ for every $i \geq i_0$ for some random, almost surely finite $i_0$. Thus, for each $i \geq i_0$, there exist at least $4^i/2$ vertices of $\Lambda_i$ that are both contained in the giant open cluster of $\Lambda_i[q]$, and have a parent in $\Lambda_{i+1}[q]$ that is contained in the giant open cluster of $\Lambda_{i+1}[q]$. Thus, conditional on this event, for each $i$ sufficiently large, the probability that the giant open cluster of $\Lambda_i[q]$ is \emph{not} connected by an open path to the giant open cluster of $\Lambda_{i+1}[q]$ is at most
\[
\left(1-q^{m(i)}\right)^{4^{i}/2} 
\leq \left(1-q i^2 4^{-i}\right)^{4^{i}/2} \leq e^{-qi^2/2},
% = 1-\left((1-q\ell^2 4^{-\ell})^{4^{\ell}/(q\ell^2)} \right)^{q\ell^2/2}
% \geq 1 - 4^{-}
\]
where we have used 
the inequality $(1-x) \leq e^{-x}$, which holds for all $x \geq 1$,
% \[\left(1-\frac{a}{x}\right)^x \leq e^{-a} \qquad \text{if } x \geq a \]
to obtain the second inequality. Since these probabilities are summable, it follows by Borel-Cantelli that there exists a random, almost surely finite $i_1 \geq i_0$ such that the giant open cluster of $\Lambda_i[q]$ is connected to the giant open cluster of $\Lambda_{i+1}[q]$ for every $i\geq i_1$. It follows that $\bbG[q]$ contains an infinite cluster almost surely. \qedhere
% this probability is at least
% \[1-e^{-q\ell^2/2}\]

\end{proof}
% It is easily verified that these choices satisfy \eqref{eq:pcnormalization}. 

\section{Nonamenability and uniqueness}

In this section we prove Theorem \ref{thm:pcpu} by constructing a nonamenable, unimodular, normalizable, bounded degree graph $G$ for which $p_c(G)=p_u(G)$ for Bernoulli bond percolation. We begin by constructing a family of partitions of the four regular tree.

\subsection{Isolated, invariantly defined partitions of the tree}

% Let $S$ be a $3$-regular tree.

\begin{figure}[p!]
\centering
\includegraphics[width=0.95\textwidth]{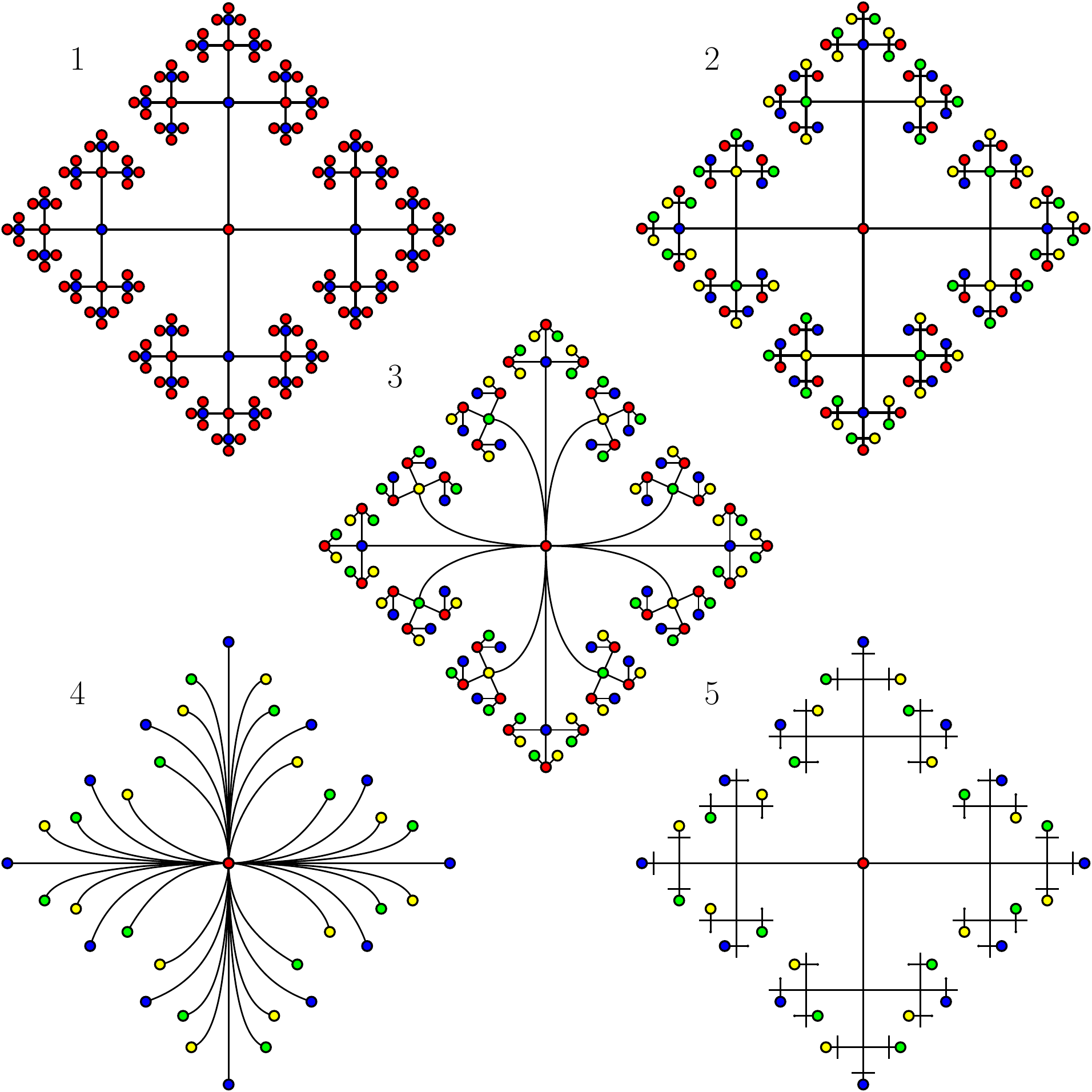}
\caption{Recursively constructing the class containing the origin in the hierarchical partition of the $4$-regular tree. $(1)$ shows the decomposition of the tree into its two bipartite classes. $(2)$ shows the subdivision of one of the two classes appearing in $(1)$ classes into four classes as occurs in $\bbV_2(S)$. $(3)$ shows the bipartite tree corresponding to the class of the origin in $\bbV_2(S)$, in which the two bipartite classes are `red' and `everything else'.
 % (here we have pictured the one corresponding to the class at the origin)
  $(4)$ shows the outcome of applying the same procedure another time, splitting the class of the origin in $\bbV_2(S)$ into four further subclasses and obtaining an associated bipartite tree for each of these classes. $(5)$ shows the classes in $(4)$ as they appear in the original $4$-regular tree.}
\end{figure}

Let $S$ be a $4$-regular tree. If we draw $S$ in the plane, then for each vertex $v$ of $S$ we obtain a cyclic ordering of the edges emanating from $v$ that encodes the clockwise order that the edges appear around $v$ in the drawing. We fix one such family of cyclic orderings, and let $\Gamma$ be the group of automorphisms of $S$ that fix this family of cyclic orderings. (In other words, we consider $S$ as a \textbf{plane tree}.)  It is well known that $\Gamma$ is unimodular.

We define the \textbf{isolation} of a subset $W$ of $V(S)$ to be the minimal distance between distinct points of $S$. 
If $A$ and $B$ are partitions of a set, we say that $A$ \textbf{refines} $B$ if every set in $W\in A$ is contained in a set of $B$. 
We say further that $A$ is a $k$\textbf{-fold refinement} of $B$ if every set in $B$ is equal to the union of exactly $k$ sets in $A$. Similarly, we say that the fold of the refinement is \textbf{bounded by $k$} if every set in $B$ is equal to the union of at most $k$ sets in $A$.

\begin{prop}
\label{prop:partition}
There exists a random  sequence of partitions $(\tilde \bbV_k)_{k\geq0}$ of $V(S)$ with the following properties.
\begin{enumerate}
\item The law of $(\tilde \bbV_k)_{k\geq 0}$ is invariant under $\Gamma$. 
\item $\bbV_0 = \{V(S)\}$, and for each $n\geq 0$, $\tilde \bbV_{n+1}$ is a refinement of $\tilde \bbV_n$ with fold bounded by $4$. 
\item Each $W \in \tilde \bbV_k$ has isolation at least $2\lfloor \log_2 k -1 \rfloor$.
\end{enumerate}
\end{prop}

The statement `the law of $(\tilde \bbV_k)_{k\geq 0}$ is invariant under $\Gamma$' should be interpreted as follows: $Gamma$ naturally acts pointwise on subsets of $V(S)$, and hence also on partitions of $V(S)$.
Then for any $\gamma\in\Gamma$, the image of $\gamma$ on the partition has the same law on the partition.
(A partition is described as a subset of $V(S)^2$, with the product $\sigma$-algebra.)
We define a tree $\bbD$, such that $V(\bbD) = \bigcup_k \tilde \bbV_k$.
The root of $\bbD$ is the trivial partition $V(S)$, and the children of a vertex in $\tilde\bbV_k$ are the included parts of $\tilde\bbV_{k+1}$.
Thus $\bbD$ has bounded degrees.

\begin{proof}
  We begin by constructing a deterministic sequence of partitions which have isolation growing linearly and fold growing exponentially.
  We then construct a random sequence of partitions that intermediate between these partitions, which will satisfy the conclusions of \cref{prop:partition}.

Let $(F_n)_{n \geq 0}$ be the sequence defined recursively by $F_0=F_1=4$ and
\[F_{n+1}=F_n(F_{n-1}-1) \qquad n\geq 1.\]
Note that this sequence grows doubly-exponentially in $n$. In particular, $F_n \leq 4^{2^n}$ for every $n\geq 0$. 
We construct a sequence of partitions $(\bbV_k)_{k\geq1}$ of $V(S)$, which we call the \textbf{hierarchical partition},  with the following properties:
\begin{enumerate}
\item $(\bbV_k)_{k\geq1}$ is $\Gamma$-invariant in the sense that for any two vertices $u,v \in V(S)$, any $\gamma \in \Gamma$ and any $k\geq 1$, if $u,v \in V$ are in the same piece of the partition $\bbV_k$ (i.e., there exists $W\in \bbV_k$ such that $u,v\in W$) then $\gamma u,\gamma v$ are also in the same piece of the partition $\bbV_k$ (i.e., there exists $W' \in \bbV_k$ such that $\gamma u, \gamma v \in W$). 
\item   
  $\bbV_1(S)$ is the partition of $V(S)$ into its two bipartite classes.
\item For each $k\geq 1$, the partition $\bbV_{k+1}(S)$ is an $F_{k-1}$-fold refinement of the partition $\bbV_k$.
	% refines $\bbV_k$ in the sense that for every $W \in \bbV_{k+1}(S)$ there exists $W' \in \bbV_k$ such that $W \subset W'$. We denote this $W'$ by $\sigma(W)$. Moreover, $W \in \bbV_k$, there are exactly 
	% four 
	% \cdot 3^{k-1}
	% sets $W'\in \bbV_{k+1}(S)$ such that $\sigma(W')=W$.
	% , where $F_n$ is the $n$th hierarchical number.
\item Each $W \in \bbV_k$ has isolation at least $2k$.
	% \item If $k\geq 2$, and $W \in \bbV_k$, then for every $v \in \sigma(W) \setminus W$ and every vertex $w$ at distance $k-1$ from $v$, there exists exactly 
	% one 
	%  $u \in W$ such that the distance between $u$ and $v$ is $2k-2$ and the geodesic from $u$ to $v$ passes through $w$. Denote this vertex $f_k(v,w,W)$.
	% , and for each vertex $u \in W$ there are exactly 
\end{enumerate}

The hierarchical partition may be constructed recursively as follows.
Suppose that  $n\geq 1$, and that $S_n$ is the plane tree whose vertices are separated into bipartite classes $V_1$ and $V_2$ such that every vertex in $V_1$ has degree $F_n$ and that every vertex in $V_2$ has degree $F_{n-1}$. We call vertices in $V_1$ \textbf{primary} and vertices in $V_2$ \textbf{secondary}. Consider a coloring of the primary vertices $V_1$ with the property that for every secondary vertex $v\in V_2$, the vertices $u_1,\ldots,u_{F_{n-1}}$ appearing in clockwise order adjacent to $v$ in $T$ have colors $1,\ldots,F_{n-1}$ up to a cyclic shift. Such a coloring is easily seen to exist and is unique up to a cyclic shifts of the colors.
 % By partitioning the primary vertices according to the residue of their color modulo $4$, then modulo $16$, and so on up to $4^{F_{n-1}}$, we obtain a sequence of $F_{n-1}$ partitions of the primary of vertices of $S_n$, invariant under the automorphisms of $S$ considered as a bipartite plane tree, such that each partition in the sequence is a $4$-fold refinement of the previous partition in the sequence.

For each color $1 \leq i \leq d_2$, let $S_{n,i}$ be the tree with vertex set $V_1$ in which two vertices are connected by an edge if and only if their distance in $S_n$ is $2$ and one of them has color $i$. This tree inherits a plane structure from $S_{n}$. Let $V_1(T_i)$ be the subset of $V(S_i)=V_1(S)$ containing the color $i$ vertices and let $V_2(S_i)$ be the subset of $V(S_i)=V(S)$ containing the vertices with color other than $i$. It is easily verified that $V_1(S_i)$ and $V_2(S_i)$ are the two bipartite classes of $S_{n,i}$ and that vertices in these classes have degrees $F_n(F_{n-1}-1)=F_{n+1}$ and $F_n$ respectively, so that $S_{n,i}$ is isomorphic to $S_{n+1}$.
Moreover, the distance in $S_{n,i}$ between any two vertices in $V_1(S_{n,i})$ is at least two less than the distance of the corresponding vertices in $S_n$: 
Indeed, it is easily verified that the distance between $u,v \in V_1(S_{n,i})$ in $S_{n,i}$ is equal to their distance in $S_n$, minus the number of vertices in $V_1(S_{n,i})$ that are included in the geodesic between $u$ and $v$ in $S_n$ (which is at least $2$ due to the endpoints being in $V_1(S_{n,i})$). See Figure \ref{fig:geodesic}.

\begin{figure}[t]
\centering
\includegraphics{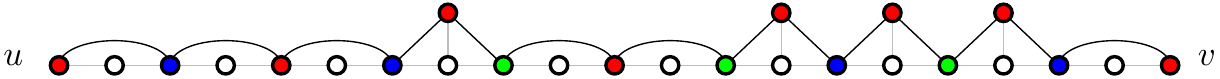}
\caption{If $u,v \in V_1(S_{n,i})$ and the geodesic between $u$ and $v$ in $S_n$ is given by the horizontal grey path, then the geodesic between $u$ and $v$ in $S_{n,i}$ is given by the black path.  Red vertices represent vertices in $V_1(S_{n,i})$, blue and green vertices represent vertices in $V_1(S_n)\setminus V_1(S_{n,i})$, white vertices represent vertices in $V_2(S_n)$. The length of the black  path is equal to the length of the grey path minus the number of red vertices. Grey curves are edges of $S_n$, black curves are edges of $S_{n,i}$.}
\label{fig:geodesic}
\end{figure}

 % suppose that $u,v \in V_1(T)$ have distance $2k$ in $T$. The geodesic between $u$ and $v$ passes through $k$ vertices of $V_2(T)$ and $k-1$ vertices of $V_1(T)$ other than $u$ and $v$. If 

%  if the geodesic connecting $u,v \in V_1(T)$ passes through $n$ vertices of $V_1(T_i)$ and $m$ vertices of $V_2(T_i)$, then the length of the path connecting $u$ and $v$ in $T_i$ is equal to 

 % (for typical pairs of vertices in $V_1(T_i)$ it will be smaller by a multiplicative constant; it will only be shorter by exactly $1$ if the vertices are joined in $T$ by a simple path that alternates between $V_1(T_i)$ and $V_2(T_i)$).

We apply this construction recursively, beginning with the $4$-regular tree $S=S_1$ separated into its two bipartite classes.  When we start step $n$ of the recursion, we
have constructed the sequence of partitions $(\bbV_k)_{k \leq n}$ and have given each $W \in \bbV_{n}$ the structure of the bipartite plane tree $S_n$ in such a way that the distance between any two vertices of $W$ in the associated copy of $S_n$ is at most their distance in $S_1$ minus $2(n-1)$. Given this data, we apply the above procedure to each of these copies of $S_n$ to complete the next stage of the recursion, obtaining a $F_{n-1}$-fold refinement $\bbV_{k+1}$ of $\bbV_k$. 
 % have partitioned the vertex set of 
% $S_1$ into 
% \[\prod_{i=0}^n 4^{F_i}\]
% copies of the bipartite plane tree $S_n$.
% After applying $n$ steps of the above procedure
 % we obtain for each $k \geq 1$ a family of bipartite plane trees whose primary and secondary vertices have degrees $4^{F_n}$ and $4^{F_{n-1}}$, respectively, and such that the primary vertices of the trees form a partition of $V(S)$. 
 This is the hierarchical partition: the above discussion implies inductively that it has the properties required above.  

It remains to modify this hierarchical partition to have bounded fold.
This is achieved by randomly adding further partitions that intermediate between $\bbV_n$ and $\bbV_{n+1}$. 
This is necessary, so that the construction in the next subsection gives a graph with unbounded degrees.

We define the \textbf{randomly intermediated hierarchical partition} $(\tilde \bbV_{k})_{k\geq0}$ of $S$ as follows.
For each $n \geq 0$, let $a_n = \lceil \log_4 F_n \rceil$, and let $b_n = \sum_{i=0}^n a_n$.
Note that $a_n \leq 2^n$ and hence $b_n \leq 2^{n+1}$ for every $n\geq 0$.
\begin{enumerate}
\item Let $\tilde \bbV_0=\bbV_0=\{V(S)\}$.
\item For each $n \geq 1$, let $\tilde \bbV_{b_n} = \bbV_n$. 
\item We construct the partitions $(\tilde \bbV_{b_n-k})_{k=1}^{a_n-1}$ recursively as follows: Given $\tilde \bbV_{b_n-k}$, for each set $W \in \tilde \bbV_{b_{n-1}}=\bbV_{n-1}$, choose uniformly at random a partition of the set $\{
  W' \in \tilde \bbV_{b_n-k} : W' \subseteq W
  \}$
  into sets that all have size four except possibly for one of the sets.
  These random choices are made independently of each other, and independent of all other randomness used in the construction. 
\end{enumerate}
The definition of $a_n$ and $b_n$ ensure that $\tilde \bbV_{n+1}$ is a refinement of $\tilde \bbV_{n}$ with fold bounded by $4$ for every $n\geq 1$.
Moreover, the sequence of random variables $(\tilde \bbV_k)_{k \geq0}$ is invariant \emph{in distribution} under $\Gamma$.
Finally, note that if we define $c_k$ to be maximal such that $b_{c_k} \leq k$ for each $k\geq 1$, then every set in $\tilde \bbV_k$ is $2c_k$ isolated for every $k\geq 0$, since every such set is contained in a set in $\bbV_{c_k}$.  Moreover, we have that 
\[
c_k \geq \lfloor \log_2k -1 \rfloor
 \]
for every $k\geq 1$.
\end{proof}

\subsection{Proof of \cref{thm:pcpu}}

We now use the randomly intermediated hierarchical partition whose existence is stipulated by \cref{prop:partition} to construct the example required by Theorem \ref{thm:pcpu}.
Let $q$ be such that $\theta_q(\Z^2)>3/4$, and let $m: \N \to \N\setminus \{0\}$ and $r:\N\to \N\setminus\{0\}$ be increasing functions such that
% We will take the functions $r$ and $m$ so that
\begin{align*}
2^{r(\ell)} \asymp \ell+1, \qquad 
% \intertext{and}
q^{m(\ell)} \asymp (\ell+1)^{-2} \quad \text{ and }  \quad 2^{2r(\ell)}q^{m(\ell)} \leq 10^{-4}.
\end{align*}
Suppose further that $r$ and $m$ have bounded increments. %in the in the sense that $r(\ell+1)\leq r(\ell)+C$ and $m(\ell+1)\leq m(\ell)+C$ for some constant $C$ and every $\ell \geq 0$.
For example, we can take
\[
r(\ell) = \left\lceil \frac{\log^+\ell}{\log 2}\right\rceil \qquad \text{ and } \qquad m(\ell) = \left \lceil \frac{2\log^+\ell + 4\log 10}{\log(1/q)}\right \rceil.
\]

Let $\bbT= \tilde{\bbT}^2(100,r)$ be the canopy tree of tori, and let $\bbG$ be obtained from $\bbT$ by replacing edges between different levels of $\bbT$ with paths of length $m$, similarly to the construction in the previous section. 
 For each vertex $v$ of $\bbG$, we write $|v|$ for the unique $\ell \geq 0$ such that either $v$ is a level-$\ell$ vertex of $\bbT$, or $v$ lies on the interior of one of the paths of length $m(\ell)$ 
 connecting level $\ell$ of $\bbT$ to level $\ell+1$ of $\bbT$ that is added when constructing $\bbG$ from $\bbT$. 

% Observe that, by induction, $|\bbV_k|=2\cdot 4^{k-1}$ if $k \geq 1$. 
% For each $v\in V$ and $k\geq0$, let $V_k(v)$ be the unique element of $\bbV_k$ containing $v$. 
Let $(\tilde \bbV_k)_{k\geq 0}$ be a random sequence of partitions of $V(S)$ satisfying the conclusions of \cref{prop:partition}, and $\bbD$ the associated tree.
Conditional on $(\tilde \bbV_k)_{k\geq 0}$, we define the graph $\bbH$ as follows.
 % be the graph with vertex set 
The vertex set $V(\bbH)$ is a subset of $\bbG\times S \times \bbD$ given by
\[
  \left \{(v_1,v_2,W) \in \bbG \times S \times \bbD \text{\ such that } |W|=|v_1| \right\}.
\]
Here, $|W|=k$ if $W\in\tilde\bbV_k$.
This construction is somewhat similar to Diestel-Leader graphs (more precisely, half of the Diestel-Leader graph), since the tree structure of $\bbG$ branches towards level $0$, and the tree $\bbD$ branches away from level $0$.

We call a vertex $(v_1,v_2,W)$ of $\bbH$ \textbf{type-1} if $v_2\in W$ and \textbf{type-2} otherwise.
We connect two vertices $(v_1,v_2,W_1)$ and $(v'_1,v'_2,W')$ of $\bbH$ by an edge if and only if one of the following hold:
\begin{enumerate}
% \item[(1).] $v_2 = v_2'$, $W=W'$, $v_1\in W$, and $v_1,v_1'$ are adjacent vertices in $\bbG(d,r,m)$ with $|v_1|=|v_1'|$, 
\item[(1).] 
  $(v_1,v_2,W_1)$ and $(v'_1,v'_2,W')$ are both type $1$, $v_2=v_2'$, and $v_1,v'_1$ are adjacent in $\bbG$, and $W,W'$ are adjacent in $\bbD$, or
% with $\left||v_1'|-|v_1'|\right|\leq 1$, or
\item[(2).]  $v_1=v_1'$, $W=W'$, and $v_2$ and $v_2'$ are adjacent in $S$.
\end{enumerate} 
We call an edge of $\bbH$ a \textbf{$\bbG$-edge} if its endpoints have the same $S$-coordinate (in which case they must both be type-$1$ vertices), and an $S$\textbf{-edge} otherwise (in which case its endpoints have the same $\bbG$-coordinate, and at least one of the vertices must be type-$2$).
Note that every connected component of the subgraph of $\bbH$ induced by the type-1 vertices (equivalently, spanned by the $\bbG$-edges) is isomorphic to $\bbG$. We call these \textbf{type-1 copies of $\bbG$ in $\bbH$.} Similarly, the type-$2$ edges span $\bbH$, and every connected component of the associated subgraph is isomorphic to the $4$-regular tree $S$. 

Let $\rho_1$ be a random root for the deterministic graph $\bbG$ chosen from the law $\mu_{\bbG}$, let $\rho_2$ be a fixed root vertex of the deterministic graph $S$, and let $W_\rho$ be chosen uniformly from $\tilde \bbV_{|\rho_1|}$. Let $\rho=(\rho_1,\rho_2,W_\rho)$. 
It follows by replacement (applied to the product $\bbG \times S$, which is unimodular and normalizable by e.g.\ \cite[Proposition 4.11]{AL07}) that $(\bbH,\rho)$ is a unimodular random rooted graph.
% Let $\rho=(v_1,v_2, W\in \bb\tilde |V|_{|v_1}}$ be a random vertex of $\bbH$ 

We call a vertex $(v_1,v_2,W)$ of $\bbH$ \textbf{type-1} if $v_2\in W$ and \textbf{type-2} otherwise.
 We call an edge of $\bbH$ 
a \textbf{$\bbG$-edge} if its endpoints have the same $S$-coordinate (in which case they must both be type-$1$ vertices), and an $S$\textbf{-edge} otherwise (in which case its endpoints have the same $\bbG$-coordinate, and at least one of the vertices must be type-$2$).
Note that every connected component of the subgraph of $\bbH$ induced by the type-1 vertices (equivalently, spanned by the $\bbG$-edges) is isomorphic to $\bbG$. We call these \textbf{type-1 copies of $\bbG$ in $\bbH$.}

Finally, given a constant $M \geq 1$, we define the graph $\tilde \bbH(M)$ by replacing each of the $S$-edges of $\bbH$ with a path of length $M$. It follows by replacement that $\tilde \bbH(M)$ can be rooted in such a way that it is a unimodular random rooted graph. Theorem \ref{thm:pcpu} therefore follows immediately from the following proposition.

\begin{proposition}
\label{prop:pcpu}
% Let $q$ be such that $\theta_q(\Z^2) > 3/4$, and let $m:\N\setminus \{0\}\to\N\setminus \{0\}$ be chosen so that
% Let
% $d(\ell)\equiv 10$, $r(\ell)=g(\ell)=\ell$, and
% \begin{align*}
% m(\ell)&=\left\lceil 
% % \frac{\log \ell + \gamma \log \log \ell}{}
% \frac{(\ell+1) \log 4 + \log 100 }{\log (1/q)}
% \right\rceil,
% \intertext{so that}
% q^{m(\ell)} &\asymp 4^{-\ell-1} \quad \text{ and } \quad q^{m(\ell)} \leq 10^{-4} \cdot 4^{-\ell-1}.
% \intertext{and}
% \end{align*}
% \]
The  random graph $\tilde\bbH(M)$ described  has bounded degrees and is nonamenable. If $M$ is sufficiently large then   
% If $M$ is sufficiently large, then 
$p_c(\tilde \bbH)=p_u(\tilde \bbH)=q$.
\end{proposition}

% For an example of a function $m$ of the form required by Proposition \ref{prop:pcpu}, we can take
% \[
% m(\ell)=\left\lceil 
% % \frac{\log \ell + \gamma \log \log \ell}{}
% \frac{(\ell+1) \log 4 + 4\log 10 }{\log (1/q)}
% \right\rceil.
% \]
The proof of Proposition \ref{prop:pcpu} will apply the notion of disjoint occurrence and the BK inequality~\cite{van1985inequalities}, see \cite[Section 2.3]{grimmett2010percolation} for background. 

\begin{proof}
% Write $\bbH = \bbH(m)$, $\tilde \bbH = \tilde \bbH(m,M)$, $\bbG=\bbG(100,2,m)$, $\bbT=\bbT^2(100,2)$, and $T=T_{100}$.
Write $\tilde \bbH=\tilde \bbH(M)$. 
% The statement concerning unimodularity and normalizability of $\tilde \bbH$ follows by applying replacement to the product $\bbG \times S$, using the automorphism invariance property of the hierarchical partition of $S$. 
 Moreover, it is immediate from the assumption that $r$ and $m$ have bounded increments that $\tilde \bbH$ has bounded degrees.

It is easily seen that stretching some edges by a bounded amount preserves nonamenability (indeed, nonamenability is stable under rough isometry), and so to prove that $\tilde \bbH$ is nonamenable it suffices to prove that $\bbH$ is nonamenable.
% To see that $\tilde \bbH$ is nonamenable, 
Observe that we can partition the vertex set of $\bbH$ into sets $\{V_i :i \in I\}$ whose induced subgraphs are copies of the $4$-regular tree $S$. 
Thus, given any finite set of vertices $K$ in $\bbH$, we can write $K = \bigcup_{i \in I} K_i$ where $K_i = K \cap V_i$. Since the subgraph of $\bbH$ induced by $V_i$ is a $4$-regular tree for every $i \in I$, it follows that the external edge boundary of $K_i$ in the subgraph induced by $V_i$ has size at least $|K_i|$ for every $i \in I$, and so we have that
\[
|\partial_E K| \geq \sum_{i \in I}|K_i| = |K|,
\] 
and hence that
\[
\frac{|\partial_E K|}{\sum_{v\in K} \deg(v)} \geq \frac{1}{\max_{v\in \bbH} \deg(v)}
\]
for every finite set of vertices $K$, so that $\bbH$ is nonamenable as claimed.

We now prove the statements concerning percolation on $\tilde \bbH$.
 It follows similarly to the proof of \cref{prop:criticalbdddegree} that $p_c(\bbG)=q$. 
% We first claim that $p_u(\bbH) \leq p_c(\bbG)$. 

\begin{lemma}
$p_u(\tilde \bbH) \leq p_c(\bbG)=q$.
\end{lemma}

\begin{proof}
  The proof is an an easy modification of the argument of Lyons and Schramm \cite[Theorem 6.12]{LS99}, and applies the main theorem of that paper as generalised to unimodular random graphs by Aldous and Lyons \cite[6.15]{AL07}; see also Theorem 6.17 of that paper.
  The sketch of the argument is as follows: It is easily verified that $\bbG$ is invariantly amenable (see e.g.\ \cite[Section 8]{AL07} and \cite{unimodular2}), so that $p_u(\bbG)=p_c(\bbG)=q$ by \cite[Corollary 6.11, 8.13]{AL07}.
  Thus, for every $p>q$, every type-$1$ copy of $\bbG$ in $\tilde \bbH$ contains a unique infinite open cluster almost surely.
  It is easy to deduce using insertion tolerance and the mass-transport principle that these clusters must all be connected to each other by open edges in $\tilde \bbH$.
  Finally, indistinguishability implies that there cannot be any \emph{other} infinite open cluster in $\tilde \bbH$, since the $\bbG$-edges within any such cluster would only have finite connected components.
\end{proof}

 To complete the proof, it suffices to show that $p_c(\tilde \bbH)\geq q$ when $M$ is sufficiently large. 
Let $V_1(\bbH)$ be the set of type-1 vertices of $\bbH$.
Define $\chi(\ell,k)$ and $\tilde \chi(\ell,k)$ for $\ell,k \geq 0$ by
\[\chi(\ell,k) = \sup_{v \in V(\bbG), |v|=\ell} \sum_{u \in V(\bbG), |u|=k} \P(v \leftrightarrow u \text{ in } \bbG[q]).\]
and
\[\tilde \chi(\ell,k) = \sup_{v \in V_1( \bbH), |v|=\ell} \sum_{u \in V_1(\bbH), |u|=k}  \P(v \leftrightarrow u \text{ in } \tilde \bbH[q]).\]
It follows easily by mass transport and insertion tolerance that every infinite cluster of $\tilde \bbH[q]$ contains infinitely many type-1 vertices of $\bbH$, and so to prove that $p_c(\tilde \bbH) \geq q$ it suffices to prove that $\sum_{k\geq0} \tilde \chi(\ell,k)<\infty$ for some (and hence every) $\ell \geq 0$.
As a first step we bound the susceptibility in $\bbG$.

\begin{lemma}\label{lem:C}
  There exists a constant $C$ such that
  \[\chi(\ell,k) \leq C (k+1)^2\]
  for every $\ell,k\geq 0$. 
\end{lemma}

Note that the choice of $M$ does not affect the definition of $\chi(\ell,k)$, and so the constant $C$ here does not depend on the choice of $M$.
A vertex of $\bbG$ at level at most $k$ has a good chance of being connected in $\bbG[q]$ to the giant component in a torus at level $k$, and therefore the dependence on $k$ cannot be improved here.

\begin{proof}
For each vertex $u$ in $\bbG$, let $t(u)$ be the associated vertex of the canopy tree $T$. Similarly, for each $x\in T$ let $\Lambda_x$ be the associated torus in $\bbT$, and let $V_x=\{ v \in V(\bbG) : x(v)=x\}$ be the associated set of vertices of $\bbG$.  
As in the proof of \cref{prop:criticalbdddegree}, if $y$ is the parent of $x$ in $T$, then the probability that $\Lambda_x$ is connected to $\Lambda_y$ by an open path in $\bbG[q]$ is at most
\[
2^{2r(|x|)} q^{m(|x|)} \leq 10^{-4}.
\]
Thus, if $x$ and $y$ are vertices of $T$ whose most recent common ancestor has height $n$, then the probability $\Lambda_x$ is connected to $\Lambda_y$ in $\bbG[q]$ is at most
$10^{-8n+4|x|+4|y|}$.

Let $x \in V(T)$, let $y$ be the parent of $x$ in $T$, and let $u \in V_x$. Then
\begin{align*}
\sum_{u \in V(\bbG): |u|=k} \P\left(v \leftrightarrow u \text{ in }\bbG[q]\right)
&\leq \sum_{n \geq \ell \vee k} \sum_{|w|=k, |w \wedge x| =n} \P\left(\Lambda_x \cup \Lambda_{\sigma(x)} \leftrightarrow \Lambda_w \cup \Lambda_{\sigma(w)} \text{ in }\bbG[q]\right) |V_w|\\
&\leq C 
\sum_{n \geq \ell \vee k} \sum_{|w|=k, |w \wedge x| =n} 10^{-8n+4k+4\ell} |V_w|\\
 &\leq C' \sum_{n\geq \ell \vee k} 10^{-8n+4\ell+4k} 10^{2n-2k} (k+1)^2\\
% &\leq C'' (k+1)^2 10^{-6(\ell \vee k)+4\ell +2k}
 &\leq 2C' (k+1)^2
\end{align*}
as claimed, where $C,C'$ are constants.
\end{proof}

We now apply \cref{lem:C} to prove that $\sum_{k \geq0} \tilde \chi(\ell,k)<\infty$. 
Observe that we may consider percolation on $\tilde \bbH$ as an inhomogeneous percolation on $\bbH$ in which every $\bbG$-edge of $\bbH$ is open with probability $q$, and every $S$-edge of $\bbH$ is open with probability $q^M$. We will work with this equivalent model for the rest of the proof.

We define a \textbf{traversal} in $\tilde \bbH$ to be a simple path in $\tilde \bbH$ that starts and ends at type-1 vertices, while every vertex in its interior is a type-2 vertex.
Observe that every traversal uses only $S$-edges, and that every simple path in $\tilde \bbH$ that starts and ends at type-1 vertices can be written uniquely as a concatenation of traversals and $\bbG$-edges.

 For each two type 1 vertices $u,v$ in $\tilde \bbH$ let $\tau(u,v)$ be the probability that $u$ and $v$ are connected by an open path. Let $\sA_i(u,v)$ be the event that $u$ and $v$ are connected by a simple open path containing exactly $i$ traversals, let $\tau_i(u,v)$ be the probability of this event, and let
\[
\tilde \chi_i(\ell,k) = \sup_{|u|=\ell} \sum_{|v|=k} \tau_i(u,v).
\] 
We have that
% \[
% \tilde \chi_0(\ell,k) = \chi(\ell,k)
% \]
% and
$\tau(u,v) \leq \sum_{i\geq0} \tau_i(u,v)$ and hence that
\[
\tilde \chi(\ell,k) \leq \sum_{i\geq0} \tilde \chi_i (\ell,k).
\]
Furthermore, $\tau_0(u,v)$ is positive if and only if $u$ and $v$ are in the same type-$1$ copy of $\bbG$, and in this case it is equal to the probability that they are connected by an open path in this copy.

Let $u,v$ be vertices of $\tilde \bbH$ with $|u|=\ell$, $|v|=k$, and let $i\geq 1$. For each type 1 vertex $w$, let $\operatorname{Tr}_w$ be the set of traversals starting at $w$. Given a traversal $t\in \operatorname{Tr}_w$, we write $t^+$ for the type-$1$ vertex at the other end of $t$.

Summing over possible choices of the $i$th traversal along a simple open path from $u$ to $v$ and applying the BK inequality, we obtain that
\begin{align}
\tau_i(u,v) &\leq  \sum_{j \geq 0} \sum_{w \in V_1(\bbH), |w|=j} \sum_{t\in \operatorname{Tr}_w } \P(\sA_{i-1}(u,w) \circ \{ t \text{ open} \} \circ \sA_{0}(t^+,v))
\nonumber
\\
&\leq \sum_{j \geq0}  \sum_{w \in V_1(\bbH), |w|=j} \sum_{t\in \operatorname{Tr}_w }  \tau_{i-1}(u,w)\tau_0(t^+,v) \P(t \text{ open}).
\label{eq:BK}
\end{align}

In $q^M$-percolation on the $4$-regular tree, the expected number of vertices that have distance at least $k$ from the root and are connected to the root by an open path is equal to
\[
\sum_{\ell \geq k} 3 \cdot 4^{k-1} (q^M)^k \leq \frac{(4q^M)^k}{1-4q^M}.
\]
Furthermore, by the isolation property of the hierarchical partition, for each type-$1$ vertex $w$ of $\bbH$, every traversal in $\operatorname{Tr}_w$ has length at least $2 \log^+|w|/\log 2$. Thus, we deduce that
 % we have that 
\begin{equation}
\label{eq:trav}
\sum_{t\in \operatorname{Tr}_w}
\P(t \text{ open}) \leq \frac{(4q^M)^{2 \log^+|w|/\log 2}}{1-4q^{M}}.\end{equation}

Thus, substituting \eqref{eq:trav} into \eqref{eq:BK} and summing over $v$, we obtain that
\begin{align*}
\sum_{|v|=k} \tau_i(u,v) 
% \sum_{j \geq0}  \sum_{w \in V_1(\bbH), |w|=j} \sum_{t\in \operatorname{Tr}_w }  q^M \tau_{i-1}(u,w) \sum_{|v|=k } \tau_0(t^+,v)
& \leq 4 \sum_{j \geq0} \tilde \chi_{i-1}(\ell,j) \chi (j,k) \frac{(4q^M)^{2 \log^+j/\log 2}}{1-4q^{M}} 
 % \frac{C'}{1-4q^M}  \sum_{j \geq0} \tilde \chi_{i-1}(\ell,j) \chi (j,k) q^{(M-1)r(j)}\\
\\&\leq \frac{C'}{1-4q^{M}}  \sum_{j \geq0} \tilde \chi_{i-1}(\ell,j) \chi (j,k) 
(j \vee e)^{-\alpha(M)}
\end{align*}
and hence that
\begin{equation}
\label{eq:induction}
\tilde \chi_i(\ell,k) \leq \frac{C'}{1-4q^{M}}\sum_{j \geq0} (j\vee e)^{-\alpha(M)} \tilde \chi_{i-1}(\ell,j) \chi (j,k),
\end{equation}
where \[\alpha(M) = \frac{2M\log(1/q)}{\log2}-4\] and $C'$ is a constant.

Take $M$ sufficiently large that
\[
% 4C q^M \sum_{j\geq0} 10^{4j} \cdot 3^j \cdot q^{Mj} 
\frac{2 C \cdot C'}{1-4q^{M}} \sum_{j\geq0} (j\vee e)^{-\alpha(M)}(j+1)^2
\leq 1/2,\] where $C'$ is the constant above, $C$ is the constant from \cref{lem:C}. 
We now prove by induction on $i$ that, for this choice of $M$,
\begin{equation}
\label{eq:induction2}
\tilde \chi_i(\ell,k) \leq C 2^{-i} (k+1)^2 
\end{equation}
for every $i\geq 0$ and $\ell,k\geq0$. The case $i=0$ follows from Lemma \ref{lem:C}. If $i\geq 1$, then \eqref{eq:induction} and the induction hypothesis yield that
\begin{align*}
\tilde \chi_i (\ell,k) &\leq \frac{2^{-i+1} C^2\cdot C'}{1-4q^{M}} \sum_{j\geq0} (k+1)^2 (j\vee e)^{-\alpha(M)}(j+1)^2,
% &\leq 4C^22^{-i+1} q^M 10^{-2\ell - 2k}\sum_{j\geq 0} 10^{4j} \cdot 3^j \cdot q^{Mj}.
\end{align*}
and our choice of $M$ yields that
\[
\tilde \chi_{i}(\ell,k) \leq C 2^{-i} (k+1)^2
\]
as claimed. This completes the proof of \eqref{eq:induction2}.

We conclude the proof by summing over $i$ and $k$ to deduce that $\tilde \chi(\ell,k)<\infty$ for every $\ell,k\geq 0$ as claimed.
\end{proof}
\subsection*{Acknowledgments}
% TH was supported by internships at Microsoft Research and a Microsoft Research PhD Fellowship.
This was was carried out while TH was a PhD student at the University of British Columbia, during which time he was supported by a Microsoft Research PhD Fellowship.

\setstretch{1}
\bibliographystyle{plain}
\bibliography{unimodular}

\end{document}